\documentclass[10pt,reqno]{amsart}
\usepackage{amsmath}
\usepackage{amssymb}
\usepackage{amsthm}

\newlength{\defbaselineskip}
\newcommand{\setlinespacing}[1]%
           {\setlength{\baselineskip}{#1 \defbaselineskip}}

\numberwithin{equation}{section}

\newtheorem{thm}{Theorem}[section]

\newtheorem{lem}[thm]{Lemma}

\theoremstyle{definition}

\theoremstyle{remark}

\numberwithin{equation}{section}

\begin{document}

\title[Local behavior of eigenfunctions]
{A note on local behavior of eigenfunctions of the Schr\"odinger operator}

\author{Ihyeok Seo}

\thanks{2010 \textit{Mathematics Subject Classification.} Primary: 34L10; Secondary: 35J10.}
\thanks{\textit{Key words and phrases.} Eigenfunctions, Schr\"odinger operator.}

\address{Department of Mathematics, Sungkyunkwan University, Suwon 440-746, Republic of Korea}
\email{ihseo@skku.edu}

\maketitle

\begin{abstract}
We show that a real eigenfunction of the Schr\"odinger operator changes sign near some point in $\mathbb{R}^n$
under a suitable assumption on the potential.
\end{abstract}

\section{Introduction}

The time evolution of a non-relativistic quantum particle
is described by the wave function $\Psi(t,x)$
which is governed by the Schr\"odinger equation
$$i\partial_t\Psi(t,x)=H\Psi(t,x),$$
where the Hamiltonian $H=-\Delta+V(x)$ is called the Schr\"odinger operator.
Here, $\Delta$ is the Laplace operator and $V$ is a potential.

The fundamental approach to find a solution of the above equation is by separation of variables.
In fact, considering the ansatz $\Psi(t,x)=f(t)\psi(x)$,
the solution can be written as $\Psi(t,x)=f(0)e^{-iEt}\psi(x)=e^{-iEt}\Psi(0,x)$, where $E$ is an eigenvalue
with the corresponding eigenfunction $\psi$ which is a solution of the following eigenvalue equation for the Schr\"odinger operator:
\begin{equation}\label{eig}
(-\Delta+V(x))\psi(x)=E\psi(x).
\end{equation}
From the physical point of view, $E\in\mathbb{R}$ is the energy level of the particle.

In this note we are interested in local behavior of $\psi$ near some point in $\mathbb{R}^n$.
By using Brownian motion ideas, it was shown in \cite{H-OH-O2S} that for a certain class of potentials $V$,
if $\psi(x_0)=0$ for $x_0\in\mathbb{R}^n$ and $\psi$ is real, then either

\smallskip

\textit{$(a)$ $\psi$ is identically zero near $x_0$ or}

\textit{$(b)$ $\psi$ has both positive and negative signs arbitrarily close to $x_0$.}

\smallskip

\noindent As remarked in \cite{H-OH-O2S}, this asserts that the nodal set $\{x:\psi(x)=0\}$ must have (Hausdorff) dimension at least $n-1$.
Also, in many cases, the first cannot occur if $\psi\not\equiv0$,
and so one can assert that the eigenfunction $\psi$ changes sign near $x_0$ in that case.

Here we will consider potentials $V$ given by
\begin{equation}\label{ks}
\|V\|:=\sup_{Q}\bigg(\int_Q|V(x)|dx\bigg)^{-1}\int_Q\int_Q\frac{|V(x)V(y)|}{|x-y|^{n-2}}dxdy<\infty,
\end{equation}
where the sup is taken over all dyadic cubes $Q$ in $\mathbb{R}^n$, $n\geq3$.
Our goal is to prove the following theorem.

\begin{thm}\label{thm}
Let $n\geq3$. Assume that $\psi\in H^1(\mathbb{R}^n)$ is real and is an eigenfunction of \eqref{eig} with $E\in\mathbb{C}$.
If the potential $V$ satisfies \eqref{ks} and $\psi$ has a zero of infinite order at $x_0\in\mathbb{R}^n$,
then either $(a)$ holds or $(b)$ holds.
\end{thm}

Let us now give more details about the assumptions in the theorem.

First, $H^1(\mathbb{R}^n)$ denotes the Sobolev space of functions
whose derivatives up to order $1$ belong to $L^2(\mathbb{R}^n)$,
and by an eigenfunction of \eqref{eig} we mean a weak solution such that
for every $\phi\in H_0^{1}(\mathbb{R}^n)$
\begin{equation}\label{weak}
\int_{\mathbb{R}^n}\nabla\psi\cdot\nabla\phi+(V(x)-E)\psi\phi\,dx=0.
\end{equation}

Next, we say that $\psi$ has a zero of infinite order at $x_0\in\mathbb{R}^n$ if for all $m>0$
\begin{equation}\label{order}
\int_{B(x_0,\varepsilon)}\psi(x)dx=O(\varepsilon^m)\quad\text{as}\quad\varepsilon\rightarrow0,
\end{equation}
where $B(x_0,\varepsilon)$ is the ball centered at $x_0$ with radius $\varepsilon$.
For smooth $\psi$, $\eqref{order}$ holds if and only if
$D^\alpha\psi(x_0)=0$ for every order $|\alpha|$.

Finally, the condition \eqref{ks} is closely related to the global Kato and Rollnik potential classes,
denoted by $\mathcal{K}$ and $\mathcal{R}$, respectively, which are defined by
\begin{equation*}
V\in\mathcal{K}\quad\Leftrightarrow\quad
\sup_{x\in\mathbb{R}^n}\int_{\mathbb{R}^n}\frac{|V(y)|}{|x-y|^{n-2}}dy<\infty
\end{equation*}
and
\begin{equation*}
V\in\mathcal{R}\quad\Leftrightarrow\quad
\int_{\mathbb{R}^3}\int_{\mathbb{R}^3}\frac{|V(x)V(y)|}{|x-y|^{2}}dxdy<\infty.
\end{equation*}
These are fundamental classes of potentials in spectral and scattering theory.
Indeed, it is not difficult to see that the Kato and Rollnik potentials satisfy the condition \eqref{ks}.
It should be also noted that there are potentials satisfying \eqref{ks} which are not in $\mathcal{K}$.
For example, $V(x)=1/|x|^2$. More generally, potentials in the Fefferman-Phong class $\mathcal{F}^p$,
which is defined by
$$V\in\mathcal{F}^p\quad\Leftrightarrow\quad
\sup_{x,r}r^{2-n/p}\bigg(\int_{B(x,r)}|V(y)|^pdy\bigg)^{1/p}<\infty$$
for $1<p\leq n/2$, satisfy \eqref{ks} (see, for example, \cite{S}).
In particular, $L^{n/2}=\mathcal{F}^{n/2}$ and even $1/|x|^2\in L^{n/2,\infty}\subset \mathcal{F}^p$ if $p\neq n/2$.
Hence the above theorem can be seen as a natural extension to potentials satisfying \eqref{ks}
of the result obtained in \cite{CG} for potentials $V\in L^{n/2,\infty}$.

\section{Preliminaries}

The key ingredient in the proof of the theorem is the following lemma,
due to Kerman and Sawyer \cite{KS} (see Theorem 2.3 there and also Lemma 2.1 in \cite{BBRV}),
which characterizes weighted $L^2$ inequalities for the fractional integral
$$I_\alpha f(x)=\int_{\mathbb{R}^n}\frac{f(y)}{|x-y|^{n-\alpha}}dy,\quad 0<\alpha<n.$$
In fact our motivation for the condition \eqref{ks} stemmed from the characterization.

\begin{lem}\label{lem}
Let $n\geq3$. Assume that $w$ be a nonnegative measurable function on $\mathbb{R}^n$.
Then there exists a constant $C_w$ depending on $w$ such that
the following inequality
\begin{equation}\label{ee}
\|I_{\alpha/2}f\|_{L^2(w)}\leq C_w\|f\|_{L^2}
\end{equation}
holds for all measurable functions $f$ on $\mathbb{R}^n$
if and only if
\begin{equation}\label{ks2}
\|w\|_\alpha:=\sup_{Q}\bigg(\int_Qw(x)dx\bigg)^{-1}\int_Q\int_Q\frac{w(x)w(y)}{|x-y|^{n-\alpha}}dxdy<\infty.
\end{equation}
Furthermore, the constant $C_w$ may be taken to be a constant multiple of $\|w\|_\alpha^{1/2}$.
\end{lem}

To prove Theorem \ref{thm} in the next section, we will use the above lemma with $\alpha=2$.
(Recall that the condition \eqref{ks} corresponds to the case $\alpha=2$ in \eqref{ks2}.)
Also, the following simple lemma is needed for handling the energy constant $E$.

\begin{lem}\label{lem2}
Let $\chi_{B(x_0,r)}$ be the characteristic function of the ball $B(x_0,r)\subset\mathbb{R}^n$.
Then there exists $r_0>0$ such that
$w=\chi_{B(x_0,r)}$ satisfies \eqref{ks2} uniformly for all $r<r_0$.
\end{lem}

\begin{proof}
First, note that
\begin{equation*}
\sup_{Q}\bigg(\int_Qw(x)dx\bigg)^{-1}\int_Q\int_Q\frac{w(x)w(y)}{|x-y|^{n-\alpha}}dxdy\leq
\sup_{x\in\mathbb{R}^n}\int_{\mathbb{R}^n}\frac{w(y)}{|x-y|^{n-\alpha}}dy.
\end{equation*}
Hence, it suffices to show that
$$\lim_{r\rightarrow0}\sup_{x\in\mathbb{R}^n}\int_{|y-x_0|<r}\frac{1}{|x-y|^{n-\alpha}}dy=0.$$
But, this is an easy consequence of the following computation:
\begin{align*}
\sup_{x\in\mathbb{R}^n}\int_{|y-x_0|<r}\frac{1}{|x-y|^{n-\alpha}}dy
&\leq\sup_{|x-x_0|<2r}\int_{|y-x_0|<r}|x-y|^{-(n-\alpha)}dy\\
&\qquad\qquad\qquad+\sup_{|x-x_0|\geq2r}\int_{|y-x_0|<r}r^{-(n-\alpha)}dy\\
&\leq\sup_{x\in\mathbb{R}^n}\int_{|x-y|<4r}|x-y|^{-(n-\alpha)}dy+Cr^\alpha\\
&\leq Cr^\alpha.
\end{align*}
\end{proof}

Finally, we recall the following lemma (see, for example, \cite{G})
concerning the doubling property.

\begin{lem}\label{lem3}
Let $f\in L_{\textrm{loc}}^1$ be a function in a ball $B(x_0,r_0)\subset\mathbb{R}^n$.
Assume that the doubling property
\begin{equation}\label{doub}
\int_{B(x_0,2r)}f\,dx\leq C\int_{B(x_0,r)}f\,dx
\end{equation}
holds for all $r$ with $2r<r_0$.
If $f\geq0$ and $f$ has a zero of infinite order at $x_0$,
then $f$ must vanish identically in $B(x_0,r_0)$.
\end{lem}

\section{Proof of Theorem \ref{thm}}
Suppose that $(b)$ does not hold.
Without loss of generality, we may then assume $\psi\geq0$ near $x_0$
(since the other case $\psi\leq0$ follows clearly from the same argument).
With this assumption, we will show that $\psi$ must vanish identically in a sufficiently small neighborhood of $x_0$,
and thereby we prove the theorem.
For simplicity of notation we shall also assume $x_0=0$,
and we will use the letter $C$ to denote positive constants possibly different at each occurrence.

Now, let $\eta$ be a smooth function supported in $B(0,2\delta)$ such that
$0\leq\eta\leq1$, $\eta=1$ on $B(0,\delta)$ and $|\nabla\eta|\leq C\delta^{-1}$.
Here, $\delta>0$ is less than a fixed $\delta_0/2$ chosen later.
Putting $\phi=\eta^2/(\psi+\varepsilon)$ with $\varepsilon>0$ in the integral in \eqref{weak},
we see that
$$\int\frac{2\eta}{\psi+\varepsilon}\nabla\psi\cdot\nabla\eta\,dx
-\int\frac{\eta^2}{(\psi+\varepsilon)^2}\nabla\psi\cdot\nabla\psi\,dx
+\int(V(x)-E)\frac{\psi\eta^2}{\psi+\varepsilon}\,dx=0.$$
(Here it is an elementary matter to check $\phi\in H_0^1$.)
By setting $\widetilde{\psi}=\ln(\psi+\varepsilon)$, it follows now that
\begin{equation}\label{equa}
\int2\eta\nabla\widetilde{\psi}\cdot\nabla\eta\,dx
-\int\eta^2\nabla\widetilde{\psi}\cdot\nabla\widetilde{\psi}\,dx
+\int(V(x)-E)\frac{\psi\eta^2}{\psi+\varepsilon}\,dx=0.
\end{equation}
Using the simple algebraic inequality
$$2ab\leq(a^2/4+4b^2),\quad a,b\geq0,$$
we bound the first integral in \eqref{equa} as follows:
\begin{align*}
\bigg|\int2\eta\nabla\widetilde{\psi}\cdot\nabla\eta\,dx\bigg|
&\leq\int2|\eta\nabla\widetilde{\psi}||\nabla\eta|\,dx\\
&\leq\int\frac14\eta^2|\nabla\widetilde{\psi}|^2dx
+\int4|\nabla\eta|^2\,dx.
\end{align*}
Then by combining this and \eqref{equa}, it is not difficult to see that
\begin{equation}\label{134}
\int\eta^2|\nabla\widetilde{\psi}|^2dx\leq\frac{16}3\int|\nabla\eta|^2dx
+\frac43\int|V(x)-E|\eta^2dx.
\end{equation}
Now, using Lemmas \ref{lem} and \ref{lem2} in the previous section,
the second term in the right-hand side of \eqref{134} is bounded as follows:
\begin{align*}
\int|V(x)-E|\eta^2dx&\leq\int|V|\eta^2dx+|E|\int\chi_{B(0,2\delta)}\eta^2dx\\
&\leq C\|V\|\int|\nabla\eta|^2dx+C|E|\|\chi_{B(0,2\delta)}\|\int|\nabla\eta|^2dx\\
&\leq C(\|V\|+1)\int|\nabla\eta|^2dx
\end{align*}
if $2\delta<\delta_0$ for a sufficiently small $\delta_0$.
Indeed, note that when $\alpha=2$ in Lemma \ref{lem}, the inequality \eqref{ee} is equivalent to
$$\int|g|^2wdx\leq C\|w\|\int|\nabla g|^2dx,\quad g\in H^1.$$
Then this and Lemma \ref{lem2} give the above bound.
Consequently, returning to \eqref{134} and recalling $\eta=1$ on $B(0,\delta)$,
we get
\begin{align}\label{578}
\nonumber\int_{B(0,\delta)}|\nabla\widetilde{\psi}|^2dx\leq\int\eta^2|\nabla\widetilde{\psi}|^2dx&\leq C\int|\nabla\eta|^2dx\\
\nonumber&\leq C\int_{B(0,2\delta)}\delta^{-2}dx\\
&\leq C\delta^{n-2}.
\end{align}

At this point, one can apply the Poincar\'{e} inequality (\cite{GT}) and the lemma of John and Nirenberg \cite{JN},
as in \cite{CG}, in order to conclude that
for some $\rho>0$
\begin{equation}\label{13}
\bigg(\frac1{|B(0,\delta)|}\int_{B(0,\delta)}e^{\rho\widetilde{\psi}}dx\bigg)
\bigg(\frac1{|B(0,\delta)|}\int_{B(0,\delta)}e^{-\rho\widetilde{\psi}}dx\bigg)<C.
\end{equation}
In fact, by the Poincar\'{e} inequality and \eqref{578},
$$\int_{B(0,\delta)}|\widetilde{\psi}-\widetilde{\psi}_B|^2dx
\leq C\delta^2\int_{B(0,\delta)}|\nabla\widetilde{\psi}|^2dx
\leq C\delta^{n},$$
where
$$\widetilde{\psi}_B=\frac1{|B(0,\delta)|}\int_{B(0,\delta)}\widetilde{\psi}\,dx.$$
Now, by H\"older's inequality
$$\int_{B(0,\delta)}|\widetilde{\psi}-\widetilde{\psi}_B|dx
\leq C\delta^{n},$$
and so $\widetilde{\psi}$ belongs to the BMO space (in $B(0,\delta_0)$).
Thus, by the lemma\footnote{\,See also Theorem 3.5 in \cite{HL} and Proposition 6.1 in \cite{T}.} of John and Nirenberg \cite{JN},
there exists some $\rho>0$ so that
$$\int_{B(0,\delta)}e^{\rho|\widetilde{\psi}-\widetilde{\psi}_B|}dx\leq C\delta^n.$$
This implies that
\begin{align*}
\int_{B(0,\delta)}e^{\rho(\widetilde{\psi}-\widetilde{\psi}_B)}dx\int_{B(0,\delta)}e^{-\rho(\widetilde{\psi}-\widetilde{\psi}_B)}dx
&=\int_{B(0,\delta)}e^{\rho\widetilde{\psi}}dx\int_{B(0,\delta)}e^{-\rho\widetilde{\psi}}dx\\
&\leq C\delta^{2n}
\end{align*}
which is \eqref{13}.
Since $\widetilde{\psi}=\ln(\psi+\varepsilon)$, by Fatou's lemma,
\eqref{13} leads to
$$\bigg(\frac1{|B(0,\delta)|}\int_{B(0,\delta)}\psi^\rho dx\bigg)
\bigg(\frac1{|B(0,\delta)|}\int_{B(0,\delta)}\psi^{-\rho}dx\bigg)<C.$$
It is a well known fact\footnote{\,The doubling property is satisfied for functions in the $A_2$ Muckenhoupt class.} (see \cite{St}, Chap. V, Section 1.5) that this implies the doubling property \eqref{doub} (with $x_0=0$) for $\psi^\rho$.
Then, by Lemma \ref{lem3}, $\psi^\rho$ must vanish identically near $x_0=0$,
and so $\psi\equiv0$ near $x_0=0$.


\end{document}